\documentclass[12pt, reqno]{amsart}
\usepackage{amsfonts,mathrsfs, amssymb}
\theoremstyle{plain}
\begingroup
\newtheorem{maintheorem}{Main Theorem}
\newtheorem{theorem}{Theorem}[section]

\newtheorem{lemma}[theorem]{Lemma}
\newtheorem{proposition}[theorem]{Proposition}
\newtheorem{corollary}[theorem]{Corollary}

\newtheorem{rmk}[theorem]{Remark}
\newtheorem{example}[theorem]{Example}
\endgroup

\theoremstyle{remark}
\begingroup
\endgroup

\mathchardef\emptyset="001F

\numberwithin{equation}{section}

\newcommand{\op}[1]{{\rm{#1}}}
\newcommand{\norm}[1]{\left\|{#1}\right\|}
\newcommand{\de}{\partial}

\newcommand{\e}{\varepsilon}

\newcommand{\R}[1]{{\mathbb R^{#1}}}

\title[Integrability result for $L^p$-vectorfields]
{An integrability result for $L^p$-vectorfields in the plane}

\author[Mircea Petrache]{Mircea Petrache}

\begin{document}

\begin{abstract}
 We prove that if $p>1$ then the divergence of a $L^p$-vectorfield $V$ on a $2$-dimensional domain $\Omega$ is the boundary of an integral $1$-current, if and only if $V$ can be represented as the rotated gradient $\nabla^\perp u$ for a $W^{1,p}$-map $u:\Omega\to S^1$. Such result extends to exponents $p>1$ the result on distributional Jacobians of Alberti, Baldo, Orlandi \cite{ABO}.
\end{abstract}
 \maketitle
\section{Introduction}
Consider a vectorfield $V\in L^p(B^2,\mathbb R^2)$. If $\op{div}V=0$ then by the Poincar\'e Lemma we know that there exists a $W^{1,p}$-function $\psi$ with 
\begin{equation}\label{reploc}
V=\nabla^\perp \psi. 
\end{equation}
The next case in which the situation is relatively standard, is when (in the sense of distributions)
\begin{equation}\label{divdirac}
\op{div} V=2\pi\sum_{i=1}^N n_i\delta_{x_i},\quad \text{for some }n_i\in \mathbb Z\setminus \{0\}\text{ and }x_i\in B^2.
\end{equation}
In this case we cannot have $V\in L^p$ unless $p<2$ (consider the model case $V(x)=\tfrac{x}{|x|^2}$, corresponding to $N=1, x_1=(0,0), n_1=1$ in \eqref{divdirac}). The representation \eqref{reploc} holds then just locally outside the points $x_i$, and the local representations do not lift to a global one. If $p>1$ then we obtain that the function $\psi$ is locally harmonic and $V$ is locally holomorphic. Therefore, it is possible to find a representation of the form \eqref{reploc} for a function $\psi\in W^{1,p}(B^2, \mathbb R/2\pi\mathbb Z)$, by taking $u=\op{Arg(V)}+C$ for any constant $C$. Equivalently, one could use the Green function for the laplacian to obtain a harmonic solution of $\nabla g=V$, and then from the regularity of $g$ the existence of $\psi$ would follow.\\
If we now consider the preimage $u^{-1}(y)$ of any regular value $y\in \mathbb R/2\pi\mathbb Z$ of $u$, then we see by Sard's theorem that this will be a rectifiable set, and with the orientation corresponding to the vectorfield $\nabla g$, we can also consider this set as an integral current $I_u$ on $\bar B^2$. The boundary of this current is precisely the sum of Dirac masses in \eqref{divdirac} (without the ``$2\pi$'' factor):
\begin{equation}\label{currentrep}
 \de I_u\llcorner B^2 = \sum_{i=1}^Nn_i\delta_{x_i} = \frac{1}{2\pi}\op{div}V.
\end{equation}
When passing to the case where we allow $N=\infty$ in \eqref{divdirac}, we have to face the new difficulty that not all the formal infinite sums of Dirac masses can be represented as the distributional divergence of an $L^p$-vectorfield. The most obvious restriction (depending on the Fubini theorem) is seen as follows. Let $\Sigma$ be a closed smooth Jordan curve and consider its perturbations $\Sigma(t), t\in[-\e,\e]$ via a family of diffeomorphisms. Then the flux $f(t)$ of $V$ through $\Sigma(t)$ should satisfy again $f\in L^p([-\e,\e])$. In particular, it cannot happen that the algebraic sum of the Dirac masses inside $\Sigma$ stays infinite for a set of times $t$ of positive measure. \\
If we assume for a moment that a rectifiable $1$-current $I$ as in \eqref{currentrep} exists, the above condition would translate by saying that the mass of the slice of $I$ along $\Sigma(t)$ is a $L^p$-function of $t$. In this work we prove a \textit{necessary and sufficient condition} for a representability property like \eqref{reploc} to hold. Consider a smooth domain $\Omega\subset \mathbb R^2$ or $\Omega= S^2\simeq \mathbb C\cup\{\infty\}$. Our main result is then:
\begin{maintheorem}[first version]\label{kessel}
 Suppose we have a vector field $V\in L^p(\Omega,\R{2})$ with $p>1$, whose divergence can be represented by the boundary of an integral $1$-current $I$ on $\Omega$, i.e.
\begin{equation}\label{divcurr1}
\frac{1}{2\pi}\int V\cdot\nabla\phi =\langle I, d\phi\rangle\quad \forall\phi\in C^\infty_c(\Omega).
\end{equation}
Then there exists a $W^{1,p}$-function $u:\Omega\to \mathbb R/2\pi\mathbb Z$ such that $V=\nabla^\perp u$ and $u|_{\de \Omega}$ has zero degree. Viceversa, for any $u\in W^{1,p}(\Omega, \mathbb R/2\pi\mathbb Z)$ with $\op{deg}(u|_{\de\Omega})=0$, the vector field $\nabla^\perp u$ belongs to $L^p$ and has divergence equal to the bounday of a current in $\mathcal I_1(\Omega)$, in the sense of \eqref{divcurr1}.
\end{maintheorem}
The zero degree condition on $\de\Omega$ in the above theorem can be removed in the following way. Consider a $L^p$-vectorfield $V$ such that 
\begin{equation}\label{eqanydegree}
 \frac{1}{2\pi}\op{div}V=\de I+\sum_{i=1}^N n_i\delta_{x_i} \quad \text{for some }n_i\in \mathbb Z\setminus \{0\}\text{ and }x_i\in \Omega.
\end{equation}
Then we can find, via the Green function method sketched in the introduction, a vectorfield $V'$ satisfying \eqref{divdirac} and a function $\psi'\in W^{1,p}(\Omega, \mathbb R/2\pi\mathbb Z)$ satisfying \eqref{reploc}, with
\begin{align*}
 \op{deg}(\psi'|_{\de\Omega})=\sum_{i=1}^Nn_i\\
  \frac{1}{2\pi}\op{div}(V-V')=\de I,
\end{align*}
and we can apply the Main Theorem to $V-V'$ obtaining a function $\psi\in W^{1,p}(\Omega, \mathbb R/2\pi\mathbb Z)$ with degree zero on $\de\Omega$ and which satisfies $\nabla^\perp\psi=V-V'$. Then $\psi+\psi'$ will satisfy 
\begin{align*}
 \nabla^\perp(\psi+\psi')=V\\
 \op{deg}((\psi+\psi')|_{\de\Omega})=\sum_{i=1}^Nn_i.
\end{align*}
With this construction we obtain the following generalization
\begin{corollary}\label{anydegree}
 Suppose we have a $L^p$-vector field $V$ satisfying \eqref{eqanydegree}. Then there exists a $W^{1,p}$-function $u:\Omega\to \mathbb R/2\pi\mathbb Z$ such that $V=\nabla^\perp u$ and $u|_{\de \Omega}$ has degree $\sum_{i=1}^Nn_i$. Viceversa, for any $u\in W^{1,p}(\Omega, \mathbb R/2\pi\mathbb Z)$ with $\op{deg}(u|_{\de\Omega})=d\in\mathbb Z$, the vector field $\nabla^\perp u$ belongs to $L^p$ and satisfies \eqref{eqanydegree}, where $d=\sum n_i$.
\end{corollary}

In the case $p=1$, a result similar to the Main Theorem above is a subcase of the result of \cite{ABO}. An equivalent statement of such result is (see also Section \ref{notations} where different notations are proposed):
\begin{proposition}[\cite{ABO}]\label{app41}
 For each integral $1$-current $I$ of finite mass on $\Omega$ there exists a map $\psi\in W^{1,1}(\Omega, \mathbb R/2\pi\mathbb Z)$ such that (in the sense of distributions)
$$
\de I= \frac{1}{2\pi}\op{div}(\nabla^\perp \psi).
$$
\end{proposition}
The distribution $\op{div}(\nabla^\perp \psi)$ is called \emph{distributional Jacobian of $\psi$}.
\begin{rmk}
 As seen in Example \ref{diffalberti}, for $p>1$, unlike the case $p=1$, a large subclass of the boundaries of integral currents is \emph{not} realized as a distributional Jacobian of any map in $W^{1,p}(B^2, S^1)$, therefore we must ask for a higher integrability condition for the current $I$: this is why the existence of the $L^p$-vectorfield $V$ is imposed.
\end{rmk}

\subsection{Different formulations of the Main Theorem}\label{notations}

We have at least three ways of looking at the manifold $S^1$, namely:
\begin{enumerate}
\item as a submanifold of $\mathbb R^2$: $S^1=\{(x,y)\in\mathbb R^2:\;x^2+y^2=1\}$,
\item via a parametrization: $S^1=\{(\cos(t), \sin(t)):\;t\in\mathbb R\}$,
\item as a quotient: $S^1=\mathbb R/2\pi\mathbb Z$.
\end{enumerate}
When considering $W^{1,p}$-maps on $B^2$ with values in $S^1$, these three points of view lead to three possible spaces:
\begin{enumerate}
\item $W_1=\{u\in W^{1,p}(B^2, \mathbb R^2):\;u_1^2(x)+u_2^2(x)=1,\text{ a. e. }x\in B^2\}$, which is just the usual definition of $W^{1,p}(B^2, S^1)$,
\item $W_2=\{(\cos(\psi), \sin(\psi)):\;\psi\in W^{1,p}(B^2,\mathbb R)\}$,
\item $W_3=\{u\in W^{1,p}(B^2,\mathbb R)\}/\sim$, where $u_1\sim u_2$ if $u_1-u_2$ is a measurable map with values on $2\pi\mathbb Z$ a.e. We denote this space by $W^{1,p}(B^2,\mathbb R/2\pi\mathbb Z)$.
\end{enumerate}
$W_1$ is isomorphic as a (topological vector space) to $W_3$ via the diffeomorphism $\phi:\mathbb R/2\pi\mathbb Z\to S^1,\; t\mapsto (\cos(t), \sin(t))$. Instead, the space $W_2$ is different from $W_1, W_3$ because of the following result:
\begin{theorem}[\cite{Demengel}]\label{dem}
 If $1\leq p <2$ and $u\in W^{1,p}(B^n,S^1)$ then the following statements are equivalent:
\begin{itemize}
 \item $u$ can be strongly approximated by smooth maps $u_k\in C^\infty(B^n, S^1)$
 \item $d(u^*\theta)=0$ in the sense of distributions
 \item There exists $\tilde u\in W^{1,p}(B^n,\mathbb R)$ such that $u=(\cos(\tilde u), \sin(\tilde u))$.
\end{itemize}
\end{theorem}
In our work, the space $W_3$ seems notationally lighter, but since $W_1$ is more common, we would like to reformulate the Main Theorem here:
\begin{maintheorem}[second version]
Let $V\in L^p(\Omega, \mathbb R^2)$ with $p>1$ be a vectorfield satisfying \eqref{divcurr1} for an integral $1$-current $I$. Then there exist a map $u\in W^{1,p}(\Omega, S^1)$ with degree zero on $\de\Omega$ such that $V=u_2\nabla^\perp u_1 - u_1\nabla^\perp u_2$. Viceversa, for any map $u\in W^{1,p}(\Omega, S^1)$ with zero degree on the boundary, the vectorfield $u_2\nabla^\perp u_1 - u_1\nabla^\perp u_2$ is in $L^p$ and has divergence equal to the boundary of an integral current.
\end{maintheorem}
We describe how to pass from the first to the second version of the Main Theorem in Section \ref{secver}.\\

Our result can be reformulated in somewhat more geometrical terms by identifying differential forms $\alpha\in L^p(\Omega, \wedge^1\Omega)$ with vectorfields
\footnote{This is a special instance of the identification of $k$-covectors $\beta$ with $(n-k)$-vectors $*\beta$ in an $n$-dimensional oriented manifold $M$ given by imposing 
$$
\langle\alpha, *\beta\rangle=\langle\alpha\wedge\beta,\vec M\rangle
$$
for all $(n-k)$-covectors $\alpha$, where $\vec M$ is an orientating vectorfield of $M$.}
 $V_\alpha\in L^p(\Omega, \mathbb R^2)$ by setting $V_\alpha=(\alpha_2, -\alpha_1)$ if $\alpha= \alpha_1\,dx + \alpha_2\,dy$, so that $d\alpha$ corresponds to $\op{div}V_\alpha$. We also observe that if we consider the tangent space of $S^1=\mathbb R/2\pi\mathbb Z$ to be identified with $\mathbb R$ in the canonical way, then $V_{u^*\theta}$ can be identified with $\nabla^\perp u$. We obtain therefore the following alternative formulation:

\begin{maintheorem}[third version]
 Let $p>1$, let $\Omega$ be either a regular open domain in $\mathbb R^2$ or the sphere $S^2$, and let $\theta$ be the volume form of $S^1$. Then the following equality holds
\begin{equation*}
 \begin{array}{c}
  \{u^*\theta:\;u\in W^{1,p}(\Omega, S^1), \op{deg}(u|_{\de\Omega})=0\}\\
=\\
\{\alpha:\;\alpha\in L^p(\Omega, \wedge^1\mathbb R^2), \exists I\in\mathcal I_1(\Omega),\;[d\alpha]=\partial I\},
 \end{array}
\end{equation*}
where $\mathcal I_1(\Omega)$ represents the finite mass integral rectifiable $1$-currents on $\Omega$ and $[d\alpha]$ is the distribution associated to $d\alpha$ by imposing 
$$
\langle[d\alpha],\varphi\rangle = \int_\Omega d\alpha\wedge \varphi\quad\forall \varphi\in \mathcal D_0(\Omega).
$$
\end{maintheorem}

\subsection{Ingredients of the proof}

The proof of the first part of our theorem follows from a density result: We prove that the class of $L^p$-vectorfields with finitely many topological singularities is dense in the class of vectorfields satisfying the condition \eqref{divcurr1}. This fact is proved in Section \ref{density}, and the proof is in the spirit of the work \cite{Bethuel1} of Bethuel (see also \cite{Bethuel2, BCDH,BCL, HL2} for related results), inspired by the ideas present in \cite{RivKess} and in \cite{kessel}. Since it is easy to approximate vectorfields with finitely many singularities, we can pass to the limit the $W^{1,p}$-maps obtained in that case in order to achieve the representation result in the first part of the Main Theorem (see Section \ref{endpf}).\\

The second part of the theorem is a direct consequence of a coarea formula (see for example \cite{MSZ}), which is related to the Sard theorem for Sobolev spaces (for which see among others \cite{BHS, DeP, Fig}). We state here just the result that we need:

\begin{theorem}\label{app5}
 If $f\in W^{1,p}_{loc}(M^m, N^n)$ for some manifolds $M,N$, then there exists a Borel representative of $f$ such that $f^{-1}(y)$ is countably $(m-n)$-rectifiable and has finite $\mathcal H^{m-n}$-measure for almost all $y\in N$ and such that for every measurable function $g$ there holds 
\begin{equation}\label{coareasob}
 \int_M g(x)|J_f(x)|d\mathcal H^m (x)= \int_N\left(\int_{f^{-1}(y)}g(x)d\mathcal H^{m-n}(x)\right)d\mathcal H^n(y),
\end{equation}
where $|J_f(x)|=\sqrt{\op{det}(Df_x\cdot Df_x^T)}$.
\end{theorem}

\section{Acknowledgements}
 I would like to thank Professor Tristan Rivi\`ere for introducing me to the topic of this paper and for the many fruitful discussions that we had on the subject.

\section{A density result}\label{density}
We start with a density result concerning vector fields as in the Main Theorem. We consider two classes of vector fields:
$$
\mathcal V_{\mathbb Z}:=\{V\in L^p(D,\R{2}):\:\eqref{divcurr1}\text{ holds}\},
$$
and
$$
\mathcal V_R:=\left\{V\in \mathcal V_{\mathbb Z}:\:V\text{ is smooth outside a finite set }S\subset D\right\}.
$$
Since $\mathcal V_{\mathbb Z}$ is closed in $L^p$, it is clear that $\overline{\mathcal V_R}^{L^p}\subset\mathcal V_{\mathbb Z}$. We want to prove the following result:
\begin{proposition}\label{density2D}
With the above notations, $\overline{\mathcal V_R}^{L^p}=\mathcal V_{\mathbb Z}$ holds.
\end{proposition}
By the remarks about $\mathcal V_R$ and $\mathcal V_{\mathbb Z}$, we just have to prove that any $V\in\mathcal V_{\mathbb Z}$ can be approximated up to an arbitrary small error $\e>0$ in $L^p$-norm, by some $V_\e\in\mathcal V_R$. The strategy of our proof is by first choosing a ``grid of circles of radius $r$'', on which we mollify appropriately $V$, and then to extend the mollified vector field inside each circle by creating finitely many singularities (which may however become unboundedly many as we let $r\to 0$), and by staying $L^p$-near the initial $V$. Finally, we will patch together the extensions on each of the balls bounded by these circles, obtaining the wanted approximant $V_\e$. The way in which we ``fill the $r$-balls'' will be by either radial or harmonic extension: we decide the method to apply depending on the degree of $V_m$ on the respective ball (we are guided in this by the result of Demengel \cite{Demengel} cited in Theorem \ref{dem}).
\subsection{Choice of a good covering}
\begin{lemma}\label{gridchoice2}
Given $r>0$, there exists a natural number $N$, a set of centers $\{x_1,\ldots,x_N\}$ and a positive measure subset $E\subset[3/4r,r]^N$ such that for all $(r_1,\ldots,r_N)\in E$
\begin{itemize}
\item The balls $\{B_1,\ldots,B_N\}$, where $B_i=B_{r_i}(x_i)$ cover $B^2$.
\item The smaller balls $B_{\frac{3}{8}r_i}(x_i)$ are disjoint.
\item For some constant depending only on $p$ and on the dimension, there holds
\begin{equation}\label{gridchoice1}
\sum_{i=1}^N\int_{\de B_i}|V\cdot n_{B_i}|^pdx\leq C_{2,p}r^{-1}||V||_{L^p(B^2)}^p,
\end{equation}
where $n_{B_i}$ is the outer normal to the ball $B_i$.
\end{itemize}
\end{lemma}
\begin{proof}
See Section \ref{proofgridchoice2}
\end{proof}
The next lemma will be needed in order to translate properties of the current $I$ to the vector field $V$.
\begin{lemma}\label{sliceofv}
Given a piecewise smooth domain $\Omega\subset B^2$, for almost all $t\in[-\e,\e]$ the following properties hold:
\begin{itemize}
\item The slice $\langle I, \op{dist}_{\de\Omega}, t\rangle$ exists and is a rectifiable $0$-current with multiplicity in $2\pi\mathbb Z$.
\item The map $\int_{\de\Omega_t}V(y)\cdot n_t(y)d\mathcal H^1(y)$ (where $n_t$ is the unit normal to $\de\Omega_t$) is well-defined and coincides with the number $\langle I,\op{dist}_{\de\Omega},t\rangle(1)\in 2\pi\mathbb Z$.
\end{itemize}
\end{lemma}
\begin{proof}
See Section \ref{proofsliceofv}
\end{proof}
Combining the Lemmas \ref{gridchoice2} and \ref{sliceofv} the following follows:
\begin{lemma}\label{gridchoice2D}
Given $r>0$, there exists a set of balls $\{B_1,\ldots,B_N\}$ with radiuses in $[3/4r,r]$ such that the thesis of Lemma \ref{gridchoice2} holds and that for any $\Omega$ which is the closure of a connected component of $B^2\setminus\cup_{i=1}^N\de B_i$ the slice $\langle I, \op{dist}_{\de\Omega}, 0\rangle$ exists, is a rectifiable $0$-current with multiplicity in $2\pi\mathbb Z$ and
$$\langle I, \op{dist}_{\de\Omega}, 0\rangle(1)=\int_{\de\Omega}V(y)\cdot n_\Omega(y)d\mathcal H^1(y)\in2\pi\mathbb Z.$$
\end{lemma}
\begin{proof}
We can use Lemma \ref{gridchoice2} first, obtaining a set $E\subset[3/4r,r]^N$. For a cover $\{B_1',\ldots,B_N'\}$ corresponding to a density point of $E$, we can then apply Lemma \ref{sliceofv} for all the closures of connected components of $B^2\setminus\cup\de B'_i$, and then consider the slices for $t\leq 0$ only.
\end{proof}
\subsection{Mollification on the boundary and estimates on good and bad balls}\label{estimates}
\begin{lemma}\label{approxongrid2D}
For a choice of balls $B_i$ as in Lemma \ref{gridchoice2D}, it is possible to find a vector field $V_m\in C^\infty(\cup_i\de B_i,\R{2})$ such that for all the regions $\Omega$ as in Lemma \ref{gridchoice2D} there holds
\begin{eqnarray}
\forall i,\;\int_{\de \Omega} V_m\cdot n_\Omega d\mathcal H^1=\int_{\de \Omega} V\cdot n_\Omega d\mathcal H^1\in 2\pi\mathbb Z\label{samedegree}\\
||V_m - V||_{L^p(\cup_i\de B_i)}\leq \e_m.\label{mollificerror}
\end{eqnarray}
\end{lemma}
\begin{proof}
Suppose $V_m$ satisfies \eqref{samedegree} and \eqref{mollificerror}, but defined only on $\cup_i\de B_i\setminus\{x:\;\exists i\neq j, x\in\de B_i\cap \de B_j\}:=\cup_i\de B_i\setminus I$. One can then easily modify it on a neighborhood of $I$ in $\cup_i\de B_i$, defining a global smooth vector field and not affecting the requirements \eqref{samedegree} and \eqref{mollificerror}.\\
We now find $V_m$ as described above. From Lemma \ref{gridchoice2D} it follows that $\sum_i \chi_{\de B_i}V\cdot n_{B_i}\in L^p(\cup_i\de B_i)$ and has integral in $2\pi\mathbb Z$. Therefore we can take its mollification as a definition of the normal component of $V_m$, automatically satisfying \eqref{samedegree} by the properties of the mollification. Then we can mollify the component of $V$ parallel to $\cup\de B_i$, and take the resulting function as the normal component of $V_m$, thereby easily verifying \eqref{mollificerror} too.
\end{proof}

\begin{lemma}\label{goodgrid2D}
Suppose $\mathcal B_n$ are families of finitely many balls which cover $B^2$ such that each point is not covered more than $C$ times and
$$
\max_{B\in\mathcal B_n}\op{diam}B \to 0\;\;(n\to\infty)
$$
Then there holds 
\begin{equation}\label{goodgrid1}
 \sum_{B\in\mathcal B_n}\|V-\bar V\|_{L^p(B)}\to0\;\;(n\to\infty).
\end{equation}
\end{lemma}
\begin{proof}
We take a smooth approximant $W=W_\e$ such that
$$
\|V-W\|_{L^p(B^2)}\leq \e/4C.
$$
Then, we can use Poincar\'e's inequality 
$$
\| W - \bar W \|_{L^{p} (B)} \leq C r_B^{1/p}\| \nabla W \|_{L^{p} (B)}, 
$$
and for $n$ big enough there will hold 
$$
\sum_{B\in\mathcal B_n}\| W - \bar W \|_{L^{p} (B)}\leq \e/2.
$$
Putting together the above two estimates, we obtain
\begin{equation*}
\begin{split}
 \sum_{B\in\mathcal B_n}&\|V-\bar V\|_{L^p(B)}\leq\\
&\leq\sum_{B\in\mathcal B_n}\|V-W\|_{L^p(B)}+\sum_{B\in\mathcal B_n}\|W-\bar W\|_{L^p(B)}+\sum_{B\in\mathcal B_n}\|\bar V-\bar W\|_{L^p(B)}\\
&\leq2 \sum_{B\in\mathcal B_n}\|V-W\|_{L^p(B)}+\e/2\\
&\leq2C\|V-W\|_{L^p(B^2)}+\e/2\\
&\leq\e,\end{split}
\end{equation*}
as wanted.
\end{proof}
We now distinguish the balls $B_i$ based on the value of the integral $\int_{\de B_i}V\cdot n_{B_i} d\mathcal H^1$: we call $B_i$ a \emph{good} ball in case such integral is zero, and a \emph{bad} ball in case it is in $2\pi\mathbb Z\setminus\{0\}$.

\begin{lemma}\label{badballsnumber}
 There exists a constant $C>0$ such that if we have a cover as in Lemma \ref{gridchoice2} with radiuses not greater than $r:=\e$, then the number of bad balls satisfies the following estimate:
$$
\#(\text{bad balls})\leq C\e^{p-2}\|V\|_{L^p}^p.
$$
\end{lemma}
\begin{proof}
 For a bad ball $B$ we have 
$$
1\leq\left|\int_{\de B}V\cdot n_B d\mathcal H^1\right|,
$$
whence we deduce successively
$$
1\leq C\e^{p-1}\int_{\de B}|V\cdot n_B|^pd\mathcal H^1
$$
and (by summing and using Lemma \ref{gridchoice2})
$$
\#(\text{bad balls})\leq C\e^{p-1}\sum_{B\text{ bad}}\int_{\de B}|V\cdot n_B|^pd\mathcal H^1\leq C\e^{p-2}\|V\|_{L^p}^p,
$$
as wanted.
\end{proof}

\begin{rmk}
We observe that by Theorem \ref{dem}, on a good ball the normal component $v_m-v:\de B_i\to\mathbb R^2$, $v_m-v=n_{B_i}[(V_m-V)\cdot n_{B_i}]$ satisfies $v_m-v=\nabla^\perp a_m$ for some $W^{1,p}$-function $a_m:\de B_i\to \mathbb R$. 
\end{rmk}
The following is a well-known result from the theory of elliptic PDEs.
\begin{lemma}\label{harmext}
Let $\tilde a$ be a function on the boundary of the unit $2$-ball $S^1$ having zero mean. Consider the harmonic extension $\tilde A$ of $\tilde a$ over $B_1$ satisfying
\begin{equation}\label{harmext1}
\left\{\begin{array}{l}
         \Delta \tilde A=0\\
         \tilde A=\tilde a\text{ on }S^1\\
        \end{array}
\right.
\end{equation}
Then the following estimate holds:
\begin{equation}\label{harmest1}
\|\nabla\tilde A\|_{L^p(B_1)}\leq C\|\nabla\tilde a\|_{L^p(S^1)}.
\end{equation}
\end{lemma}

We will consider $a'_m$ on the boundary $\de B$ of a small ball instead of $\tilde a$ on $\de B_1$, and obtain a harmonic extended function, denoted by $A'_m$, satisfying the analogous of \eqref{harmext}. Taking into account the scaling factors we then obtain the following estimate analogous to \eqref{harmest1} (where $r$ is the radius of $B$):
\begin{equation}\label{harmest2}
\|\nabla A'_m\|_{L^p(B)}\leq Cr^{1/p}\|v_m-v\|_{L^p(\de B)}.
\end{equation}
We claim that extending $V_m:=\nabla^\perp A'_m +\bar V$ inside $B$, we obtain the wanted approximation:
\begin{lemma}\label{goodballs}
If $B$ is a \emph{good} ball of radius $\e$ on whose boundary we have $\norm{V - V_m}_{L^p(\de B)}<\e$, then the extended smooth vector field $V_m$ defined as above satisfies on $B$
$$
\norm{V-V_m}_{L^p(B)}\leq C\e^{\frac{p-1}{p}}\|v_m-v\|_{L^p(\de B)}+ \norm{V-\bar V}_{L^p( B)}. 
$$
\end{lemma}

\begin{proof}
We can then write
$$
\|V- V_m\|_{L^p(B)}\leq \|V-\bar V\|_{L^p(B)} + \|\nabla^\perp A'_m\|_{L^p(B)}.
$$
The second term above is estimated as in \eqref{harmest2}, by $C\e^{1/p}\|v_m-v\|_{L^p(\de B)}$, and the estimate \eqref{mollificerror} gives then $\e\|v_m-v\|^p_{L^p(\de B)} \leq C\e^{p-1}$, finishing the proof.
\end{proof}

\begin{lemma}\label{badballs}
If $B\subset B_1^2$ is a \emph{bad} ball of radius $\e$ and $v_m$ is the smooth orthogonal vector field on $\de B$ related to $V_m$ as in Lemma \ref{approxongrid2D} and $V'_r$ is the radial extension $V'_r(\theta, \rho):=\frac{\e}{\rho}v_m(\theta)$ (in polar coordinates centered in the center of $B$), then with the notation $V_r:=V'_r - \bar V$, we have the estimate:
$$
\norm{V-V_r}_{L^p(B)}\leq \norm{V-\bar V}_{L^p(B)}+C\e.
$$
\end{lemma}
\begin{proof}
There holds
\begin{eqnarray*}
\norm{V-V_r}_{L^p(B)}&\leq&\norm{V-\bar V}_{L^p(B)}+\norm{V'_r}_{L^p(B)},\\
\norm{V'_r}_{L^p(B)}^p&=&\int_0^\e\int_0^{2\pi}\left(\frac{\e}{\rho}\right)^p|v_m(\theta)|^pd\theta\rho d\rho\\
&=&C\e^2\|v_m\|^p_{L^p(\de B)}.
\end{eqnarray*}
From \eqref{gridchoice1}, \eqref{mollificerror} and the last equality above we conclude that $\|V'_r\|^p_{L^p(B)}\leq C\e^p$, as wanted.
\end{proof}
\subsection{End of proof of Proposition \ref{density2D}}

\subsubsection{Application of Lemmas \ref{goodballs} and \ref{badballs}}
We will use the results of Section \ref{estimates} in order to achieve a first global approximation $V_1$ of $V$. We again start with the ball $B_1$, where we will use Lemma \ref{goodballs} or \ref{badballs}, respectively when $B_1$ is a good or a bad ball. The new vector field $V_1$ obtained by replacing $V$ with the so obtained local approximant satisfies the following properties:
\begin{itemize}
\item \textbf{Good approximation of $V$ on $B_1$:} The approximation error in $L^p$-norm on the ball $B_1$ is bounded above by $ C\e^{1/p} + \norm{V-\bar V}_{L^p(B_1)}$.
\item \textbf{Controlled behavior on the boundary:} The extension inside $B_1$ is equal to $\nabla^\perp A'_m+\bar V$ on the boundaries of the $B_i$'s, and in particular it has degree equal either to the one of $V_m$ or to zero  on any of the boundaries of the domains $\Omega$ of Lemma \ref{approxongrid2D}. Indeed, $A'_m$ is smooth, so $V_m|_{B_1}$ will have divergence either zero (for good balls) or a Dirac mass in the center of $B_1$ (for bad balls), while on $B_1\setminus\cup_i\de B_i$, $V_m=V_1$. Therefore $V_1$ also has the properties stated in Lemma \ref{approxongrid2D}. 
\end{itemize}
This allows us to apply iteratively the above construction for the balls $B_j,j=2,\ldots, N$, in order to further modify $V_1$ obtaining successively approximants $V_2,\ldots, V_N$ according to Lemmas \ref{goodballs}, \ref{badballs}, and we are able to continue ensuring the smallness condition $\|V-V_m\|_{L^p(\de B_j)}$.
\begin{lemma}\label{errorest}
 For each $\bar \e>0$ there exist a radius bound $\e$ and an approximation error bound $\e_m$ (in Lemma \ref{approxongrid2D}) such that the approximant $V_N$ constructed above satisfies
$$
\|V-V_N\|_{L^p(B^2)}\leq\bar \e.
$$
\end{lemma}
\begin{proof}
By Lemmas \ref{goodballs} and \ref{badballs} we can estimate
\begin{eqnarray*}
\|V-V_N\|_{L^p(B^2)}&\leq&\sum_{\text{good }B}\left[C\e^{\frac{p-1}{p}}\|v_m-v\|_{L^p(\de B)}+ \norm{V-\bar V}_{L^p( B)}\right] \\
&&+\sum_{\text{bad }B}\left[\norm{V-\bar V}_{L^p(B)}+C\e\right]\\
&=&\sum_{\text{all }B}\norm{V-\bar V}_{L^p(B)} + C\e\#(\text{bad balls})+C\e^{\frac{p-1}{p}}\e_m.
\end{eqnarray*}
Consider now the expression in the last row above: the first term converges to zero by Lemma \ref{goodgrid2D}, and the last one is small for $\e_m$ small. The middle term can be estimated using Lemma \ref{badballsnumber} and has thus a bound of the form $C\e^{p-1}\|V\|_{L^p}^p$. Since $p>1$ and $V\in L^p$, also this term is small for $\e$ small.
\end{proof}

\subsubsection{Smoothing on the boundary}
The preceding iteration procedure gives us an $L^p$-approximant with error $C\e$ if the radius $r$ of the balls was chosen to be equal to $\e$. Moreover it is easy to verify that 
\begin{equation}\label{divergencef1}
\op{div}V_N=\sum_{i=1}^N\delta_{x_i}\int_{\de B_i}V_m,\quad\text{locally outside}\cup_i\de B_i
\end{equation}
where $x_i$ is the center of $B_i$. The resulting vector field $V_N$ is however not in $\mathcal V_R$: for instance, it is not smooth on all of $\cup_i\de B_i$. We will thus smoothen $V_N$ as follows. We observe that locally near $\cup_i \de B_i$ on $B^2\setminus\cup_i\de B_i$, $V_N$ is represented as $\nabla^\perp A_i:=\nabla^\perp A'_i+\bar V_i$, where $A'_i$ is smooth and $\bar V_i$ is a constant equal to the average of $V$ on a particular $B_i$. We can take an open cover by small balls of a neighborhood of $\cup_i \de B_i$ then mollify the functions $A_i$ inside each of these small balls, then use a partition of unity to patch the mollifications into a single smooth function $A_\e$ without losing more than an error of $\e$ in $L^p$-norm. Then we can safely define $V_\e:=\nabla^\perp A_\e$. $\square$

\section{Proof of Theorem \ref{kessel}}\label{endpf}
\begin{proof}
We first show how to deduce the second part of Main Theorem \ref{kessel} from Proposition \ref{density2D}. \\

The main idea is that, by Proposition \ref{density2D}, we can take a sequence $V_n\stackrel{L^p}{\to}V$ which belongs to $\mathcal V_R$ and construct $u_n$'s such that $V_n=\nabla^\perp u_n$, and they will be constrained to converge to a $u$ with the wanted property $\nabla^\perp u=V$. We remark that if $V_n$ is smooth and divergence-free outside a discrete set $\Sigma$, then $V_n^\perp$ is locally holomorphic, and the fact that the divergence around any point of $\Sigma$ is a Dirac mass with coefficient in $2\pi\mathbb Z$ translates into saying that $V^\perp_n$ has degree equal to that coefficient around that point. Consider the divisor $D$ supported on $\Sigma$ with residue corresponding to the divergence of $V_n$. Therefore $V_n^\perp$ is a meromorphic function with divisor $D$, so we can take $u_n:=\op{arg}V_n^\perp$, which is well-defined with values in $\mathbb R/2\pi\mathbb Z$ and satisfies $\nabla u_n=V_n^\perp$.\\
We have thus functions $u_n\in W^{1,p}(\Omega, S^1)$ satisfying $V_n=\nabla^\perp u_n$ and therefore $\nabla  u_n\stackrel{L^p}{\to}V$. We can change the $u_n$ by a constant so that $\tfrac{1}{|\Omega|}\int_\Omega u_n = 0\in\mathbb R/2\pi\mathbb Z$. Then by Poincar\'e's inequality we have that $u_n$ form a $L^p$-Cauchy sequence, converging therefore to $\bar u\in L^p(\Omega, \mathbb R/2\pi\mathbb Z)$. After extracting a subsequence $u_n\stackrel{W^{1,p}}{\rightharpoonup}u\in W^{1,p}$. Since we have a.e.-convergence too, it must hold $u=\bar u$ and $\nabla^\perp u=V$, as wanted.\\

As above, $u_n\stackrel{W^{1,p}}{\to}u$ and $d(u_n^*\theta)$ are finite sums of Dirac masses with integer coefficients. 
 The fact that for $u\in W^{1,p}(\Omega,\mathbb R/2\pi\mathbb Z)$ the vectorfield $\nabla^\perp u$ has the properties required from the vectorfield $V$ in the theorem follows from Theorem \ref{app5}, by taking 
$$
I=I^u_z=\tau\left(u^{-1}(z), 1, \frac{\nabla^\perp u(x)}{|\nabla^\perp u(x)|}\right),
$$
for a common regular value $z\in \mathbb R/2\pi\mathbb Z$ of all the $u_n$ and of $u$. Then by the coarea formula (observing that in our case $|J_u|=|\nabla^\perp u|$) we have for all $f\in C^\infty_c(\Omega)$
\begin{eqnarray*}
 \int_\Omega u^*\theta\wedge df&=&\int_\Omega \nabla^\perp u\cdot \nabla f dx = \int_{S^1}dy\int_{u^{-1}(y)} \left\langle df, \frac{\nabla^\perp u}{|\nabla^\perp u|}\right\rangle d\mathcal H^1\\
&=&\int_{S^1}I^u_y(df)dy = \frac{1}{2\pi}\int_{S^1}\de I^u_y(f)dy.
\end{eqnarray*}
Similarly we obtain for all $n$:
$$
\int_\Omega u_n^*\theta\wedge df=\int_{S^1}\de I^{u_n}_y(f)dy=2\pi \de I^{u_n}_z(f),
$$ 
since for $u_n$ with finitely many singularities $\de I^{u_n}_y(f)$ does not depend on $y$. We have (since $C^\infty_c\subset (W^{1,p})^*$)
$$
\int_{S^1}\de I^{u_n}_y(f)dy\to\int_{S^1}\de I^u_y(f)dy,
$$
We may assume that the integrands converge pointwise in $z$, proving the condition \eqref{divcurr1}, and thus finishing the proof.
\end{proof}

\subsection{Proof of the second version of the Main Theorem}\label{secver}

\begin{proof}
We consider the diffeomorphism $\varphi:\mathbb R/2\pi\mathbb Z\to S^1\subset\mathbb R^2$ given by $t\mapsto (\cos t, \sin t)$, and then instead of the map $u:\Omega\to\mathbb R/2\pi\mathbb Z$ obtained in the Main Theorem \ref{kessel} we take the map $\bar u:=\varphi\circ u:\Omega\to S^1\subset \mathbb R^2$. We then obtain
\begin{eqnarray*}
 \nabla \bar u&=&\nabla u\otimes (\nabla\varphi\circ u)\\
&=&\left(\begin{array}{cc} -\de_1 u \sin u & \de_1u \cos u\\ -\de_2 u \sin u& \de_2 u \cos u\end{array}\right),
\end{eqnarray*}
therefore 
$$
\bar u_1\nabla^\perp \bar u_2 - \bar u_2\nabla^\perp \bar u_1 = \cos^2 u \left(\begin{array}{c} -\de_2 u\\ \de_1 u \end{array}\right) + \sin^2 u \left(\begin{array}{c} -\de_2 u\\ \de_1 u \end{array}\right)  = \nabla^\perp u.
$$
This proves the wanted identifications, and we only need to prove that if $\bar u\in W^{1,p}(\Omega, S^1)$ then $\bar u_1\nabla^\perp \bar u_2 - \bar u_2\nabla^\perp \bar u_1 \in L^p(\Omega, \mathbb R^2)$. This follows using the relation $\bar u_1^2+\bar u_2^2=1$ and its consequence $\bar u_1\nabla^\perp \bar u_1 = -\bar u_2\nabla^\perp \bar u_2$. We have indeed:
\begin{eqnarray*}
 |\bar u_1\nabla^\perp \bar u_2 - \bar u_2\nabla^\perp \bar u_1|^2&=& \bar u_1^2|\nabla^\perp \bar u_2|^2 -2\bar u_1 \bar u_2 \nabla^\perp \bar u_2\nabla^\perp \bar u_2 +\bar u_2^2|\nabla^\perp \bar u_1|^2\\
&=&(\bar u_1^2 +\bar u_2^2)|\nabla^\perp \bar u_2|^2 + (\bar u_1^2 +\bar u_2^2)|\nabla^\perp \bar u_1|^2\\
&=& (\de_2 \bar u_2)^2 + (\de_1 \bar u_2)^2 + (\de_2 \bar u_1)^2 + (\de_1 \bar u_1)^2\\
&=&|\nabla \bar u|^2,
\end{eqnarray*}
and since $u\in W^{1,p}$, this proves the result.
\end{proof}

\section{Proof of Proposition \ref{gridchoice2}}\label{proofgridchoice2}

Our aim here is to prove the following
\begin{proposition}\label{gridchoice}
 Given $r>0$, there exists a cover of $B_1^2$ by a finite set of balls $\{B_r(y_1),\ldots,B_r(y_N)\}$ such that the balls $B_{r/2}(y_i)$ are disjoint and such that for some constant depending only on $p$ and on the dimension,
\begin{equation}\label{gridchoice11}
 \sum_{i=1}^N\int_{\de B_r(y_i)}|V\cdot n_{B_r(y_i)}|^pdx\leq C_{2,p}r^{-1}||V||_{L^p(B^2)}^p,
\end{equation}
where $n_{B_r(y_i)}$ is the outer unit normal vector to the circle $\de  B_r(y_i)$.
\end{proposition}
Directly form the proof of Proposition \ref{gridchoice} we can also obtain the more refined result:
\begin{proposition}
 Given $r>0$, there exists a natural number $N$, a set of centers $\{x_1,\ldots,x_N\}$ and a positive measure subset $E\subset[3/4r,r]^N$ such that for all $(r_1,\ldots,r_N)\in E$ 
 \begin{itemize}
 \item The balls $\{B_1,\ldots,B_N\}$, where $B_i=B_{r_i}(x_i)$ cover $B^2$.
 \item The smaller balls $B_{\frac{3}{8}r_i}(x_i)$ are disjoint.
\item For some constant depending only on $p$ and on the dimension, there holds
\begin{equation*}
 \sum_{i=1}^N\int_{\de B_i}|V\cdot n_{B_i}|^pdx\leq C_{2,p}r^{-1}||V||_{L^p(B^2)}^p.
\end{equation*}
\end{itemize}
\end{proposition}
\subsection{Equivalent definition of the pointwise norm of $V$}
 $\langle V,\theta\rangle$ for a vector $\theta\in S^1\subset\mathbb R^2$, can be expressed as $|V||\cos\gamma|$ where $\gamma$ is the angle between $\theta$ and $V$. After noting 
$$
\int_{S^1}|\cos\gamma|^pd\theta=:c_p,
$$
we can write 
\begin{equation}\label{equivnorm1}
|V|^p=\frac{1}{c_p}\int_{S^1}|\langle V,\theta\rangle|^pd\theta.
\end{equation}
We now pass to consider the circle $S_r(x)=\de B_r(x)$. Then we can write 
$$
\int_{S_r(x)}V(y)\cdot n_{B_r}(y)\;dy=\int_{S_r(x)}\left\langle V(y), \left(\frac{y-x}{|y-x|}\right)\right\rangle dy=\int_{S^1}\langle V(x+r\theta), \theta\rangle rd\theta.
$$
Given a positive number $r$, a point $x\in\R{2}$ then belongs to $S_r(y)$ exactly for $y\in S_r(x)$, and we have by \eqref{equivnorm1}, that\\
\begin{eqnarray}
\int_{S_r(x)}\left|V(x)\cdot n_{B_r(y)}(x)\right|^p\;dy&=&
\int_{S_r(x)}\left|\left\langle V(x), \left(\frac{x-y}{|x-y|}\right)\right\rangle\right|^p\;dy \nonumber\\
&=&\int_{S^1}|\langle V(x), \theta\rangle|^p rd\theta\nonumber\\
&=&c_pr|V(x)|^p. \label{equivnorm2}
\end{eqnarray}

\subsection{Proposition \ref{gridchoice} and an extension of it}
\begin{proof}[Proof of Proposition \ref{gridchoice}:]
We observe that \eqref{equivnorm2} can be integrated on $\R{2}$ (after having extended $V$ by zero outside $B^2$), to give
\begin{eqnarray}
 c_pr\int_{B^2}|V(x)|^p\;dx
&=&c_pr\int_{\R{2}}|V(x)|^p\;dx \nonumber\\
&=&\int_{\R{2}}\int_{S_r(x)}\left|V(x)\cdot n_{B_r(y)}(x)\right|^p\;dy\;dx \nonumber\\
&=&\int_{\R{2}}\int_{S_r(z)}\left|V(x)\cdot n_{B_r(z)}(x)\right|^p\;dx\;dz \nonumber\\
&=&\int_{B^2_{1+r}}\int_{S_r(z)}\left|V(x)\cdot n_{B_r(z)}(x)\right|^p\;dx\;dz.\label{lpnorm1}
\end{eqnarray}
We now define some systems of disjoint balls. We consider a set
\begin{equation}\label{eq:disjballs}
S=\{x_1,\ldots,x_N\}\subset B^2_{1+r}\text{ s.t. }\left\{
\begin{array}{l}
\min_{1\leq i\neq j\leq N}d(x_i,x_j)\geq r \\ 
S\text{ is maximal}
\end{array}
\right.
\end{equation}
and the corresponding set of translates of the ball $B_r(0)$.
$$
\mathcal S:=S+B_r(0)=\left\{\{x_1+y,\ldots,x_N+y\}:\;y\in B_r(0)\right\}
$$
Then $\mathcal S$ covers $B_{1+r}$ (by maximality in the definition of $S$) at most $C$ times, where $C$ is a packing number (by the requirement on the mutual distances of elements of $S$). We can then bound the integral \eqref{lpnorm1} from below as follows
\begin{eqnarray*}
 c_pr\int_{B^2}|V(x)|^p\;dx
&=&\int_{B^2_{1+r}}\int_{S_r(z)}\left|V(x)\cdot n_{B_r(z)}(x)\right|^p\;dx\;dz\\
&\geq& \frac{1}{C}\int_{B_r}\left(\sum_{i=1}^N\int_{S_r(x_i+z)}|V\cdot n|^pdy\right)dz
\end{eqnarray*}
and it follows that there exists $z\in B_r$ such that 
\begin{equation*}
 \sum_{i=1}^N\int_{S_r(x_i+z)}|V\cdot n|^pdy\leq\frac{Cc_pr}{|B_r|}\int_{B^3}|V|^p\;dx=C_{2,p}r^{-1}||V||_{L^p(B^2)}^p.
\end{equation*}
 This is enough to prove \eqref{gridchoice11}. Moreover, again by the maximality of $S_0$, the balls $\{B_r(x_i+z)\}_{i=1}^N$ cover $B_1^2$, and by the requirement on the distances of the centers in \eqref{eq:disjballs}, the $B_{r/2}(x_i+z)$ are disjoint, proving Proposition \ref{gridchoice}.
\end{proof}

\section{Proof of Proposition \ref{sliceofv}}\label{proofsliceofv}
We suppose here that we are given a vector field $V\in L^p(B^2,\mathbb R^2)$, for some $p\neq\infty$, such that for some integer multiplicity rectifiable current $I$ we have $\op{div}V=\de I$. This means more precisely that
\begin{equation}\label{divcurr}
\int V\cdot\nabla\phi =\langle I, d\phi\rangle\text{, for all functions }\phi\in C^\infty_0(B^2).
\end{equation}
Here $\langle I, d\phi\rangle$ refers to the action of the current $I$ on the $1$-form $d\phi$.
If $\Omega$ is a piecewise smooth domain, we will also call $\de\Omega_t$ the set $\{x\text{ s.t. }\op{dist}_{\de\Omega}(x)=t\}$. By $\op{dist}_{\de\Omega}$ we here denote the \emph{oriented} distance from $\de\Omega$, i.e. the function defined on a small neighborhood of $\de\Omega$ and equal to $\op{dist}_\Omega$ outside $\Omega$ and to $-\op{dist}_{\Omega^c}$ inside $\Omega$. Our aim in this section is to prove the following
\begin{proposition}
 Given a piecewise smooth domain $\Omega\subset B^2$, for almost all $t\in[-\e,\e]$ the following properties hold:
 \begin{itemize}
\item The slice $\langle I, \op{dist}_{\de\Omega}, t\rangle$ exists and is a rectifiable $0$-current with multiplicity in $2\pi\mathbb Z$.
\item The map $\int_{\de\Omega_t}V(y)\cdot n_t(y)d\mathcal H^1(y)$ (where $n_t$ is the unit normal to $\de\Omega_t$) is well-defined and coincides with the number $\langle I,\op{dist}_{\de\Omega},t\rangle(1)\in 2\pi\mathbb Z$.
 \end{itemize}
 \end{proposition}
\begin{proof}
We consider a family of symmetric mollifiers $\varphi_\e:\mathbb R\to\mathbb R^+$ supported in $[-\e,\e]$, and their primitives $\chi_\e(x):=\int_{-\infty}^x\varphi_\e dt$. We will consider a non negative function $g$ which is $C^\infty_c$-extensions to a neighborhood of $\de \Omega$ of the constant function equal to $1$ on all the $\Omega_t$'s with $t\in[-2\e,2\e]$, and we write the current $I$ as $(\mathcal M_I,\theta_I,\tau_I)$, where $\mathcal M_I$ is a $1$-rectifiable set supporting the current $I$, $\tau_I$ is the orienting vector of $I$ and $\theta_I$ is the multiplicity of $I$. Then the currents approximating the slice $\langle I, f, t\rangle$ (for some Lipschitz function $f:B^2\to[-2\e,2\e]$), when it exists, satisfy:
\begin{eqnarray}
I\llcorner f^\#(\varphi_\e(\cdot-t) d\tau)(g)&=&\int_{\mathcal M_I}\langle \tau_I(x), g(x)\varphi_\e(f(x)-t) df_x\rangle d\mathcal H^1(x)\label{convcurr}\\
&=& 
\int_{\mathcal M_I}\langle \tau_I(x), g(x)d(\chi_\e(\cdot - t)\circ f)_x\rangle d\mathcal H^1(x)\nonumber\\
&=&
\langle I\llcorner g,dF_\e\rangle\quad\text{where}\quad F_\e(x):=\chi_\e(f_\e(x)-t),\nonumber\\
&=&
\langle I,dF_\e\rangle\quad\text{since}\quad \op{spt}F_\e\subset\{g=1\},\nonumber\\
&=&
\int_{\{x:\:|f(x)-t|\leq\e\}} V\cdot \nabla F_\e dx^2\quad\text{(by \eqref{divcurr})}.\nonumber
\end{eqnarray}
Now we take $f(x):=\op{dist}_{\de\Omega}(x)$, obtaining that a.e. on a tubular neighborhood 
\begin{equation*}
 T(\Omega,2\e):=\cup_{-2\e\leq t\leq2\e}\de\Omega_t,
\end{equation*}
$\nabla f$ exists, and on each $\de\Omega_\tau=\{f=\tau\}$ it is a.e. equal to the unit normal vector $n_\tau$. Therefore we have
\begin{eqnarray*}
\nabla F_\e(x)&=&\nabla(\chi_\e(\cdot - t)\circ f_\e)(x)\\
&=&\varphi_\e(\cdot - t)\circ f(x) \nabla f(x) =\left[\varphi_\e(\cdot - t)\circ\op{dist}(x,\de\Omega) \right]n_{\op{dist}_{\de\Omega}(x)}
\end{eqnarray*}
and
\begin{eqnarray*}
\int_{\{|f-t|\leq\e\}} V\cdot \nabla F_\e dx^2
&=&\int_{T(\Omega,2\e)}\varphi_\e\circ\op{dist}_{\de\Omega_t}(x)\:V(x)\cdot \nabla(\op{dist}_{\de\Omega_t})(x)dx^2\\
&=&\int_{-\e}^\e\varphi_\e(t)\left(\int_{\de\Omega_t}V\cdot n_t d\mathcal H^1\right) dt.
\end{eqnarray*}
As in the usual theory of slicing, for almost all $t$'s the currents $I\llcorner f^\#(\varphi_\e(\cdot - t) d\tau)$ converge weakly to the slice $\langle I, f, t\rangle$ as $\e\to0$. Similarly, $V$ being in $L^p$, a dominated convergence argument gives also for almost all $\bar t$ the convergence 
\begin{equation}\label{convdeg}
\int_{-\e}^\e\varphi_\e(\tau - t)\left(\int_{\de\Omega_\tau}V\cdot n_{\de\Omega_\tau}d\mathcal H^1\right) dt\to\int_{\de\Omega_t}V\cdot n_{\de\Omega_{t}}d\mathcal H^1.
\end{equation}
The fact that almost all slices of an integer multiplicity rectifiable current are integer multiplicity rectifiable gives the first point of the Proposition, while the second point follows from \eqref{convcurr} and \eqref{convdeg}.
\end{proof}

\section{Further remarks concerning the Main Theorem}
We want first to point out that not all boundaries of rectifiable integral currents $\de I$ are representable as $u^*\theta$ for $u\in W^{1,p}(\Omega, S^1)$, if $p>1$, showing that this case is more subtle than the case $p=1$ treated in Theorem \ref{app41}. To do this, we use the second formulation of the Main Theorem, which says that such $u^*\theta$ would then be equal to $V^\perp$ for some vectorfield $V\in L^p$ satisfying $\op{div}V=\de I$. We will show that not all integral currents $I$ have $\de I$ equal to a divergence of a $L^p$-vectorfield.\\
 Suppose first that we have a vectorfield $V$ on $B_\e(p)$ satisfying $\op{div}V=\delta_p$ (where $\delta_p$ is the Dirac mass in $p$). Then for almost all $r\in[0,\e[$ we have 
\begin{equation}\label{degree1}
 \int_{\de B_r(p)}V\cdot n_{B_r(p)}\; d\,\mathcal H^1= 1,
\end{equation}
and we see that under the constraint \eqref{degree1}, the minimal $L^p$-mass is achieved by the radial (in polar coordinates around $p$) vectorfield 
$$
 V_{min}(\theta, r)=\frac{1}{2\pi r}\hat r
$$
(by a rearrangement argument and by the convexity of the $L^p$-norm for $p>1$). We therefore obtain (for some geopmetric constant $C$)
\begin{equation}\label{lpnormestimate}
 \|V\|_{L^p(B_\e(p))}^p\geq  \|V_{min}\|_{L^p(B_\e(p))}^p=C\e^{2-p}
\end{equation}
We see that such estimate on the norm of $V$ is only dependent on the fact that $(\op{div}V)\llcorner B_\e(p)=\delta_p$. We can now use a series of inequalities like \eqref{lpnormestimate} on a series of (disjoint) balls in order to find our counterexample.

\begin{example}\label{diffalberti}
 Take a sequence of positive numbers $(a_i)_{i\in\mathbb N}$ such that
\begin{eqnarray}
 \sup_i a_i&=&\e\nonumber\\
 \sum_{i=1}^\infty a_i&=&2\label{mass1}\\
 \sum_{i=1}^\infty a_i^{2-p}&=&=+\infty.\label{infinitest}
\end{eqnarray}
It is possible to achieve this for any $\e>0$, since $p>1$.\\

Now take a $2$-dimensional domain $\Omega$. It is possible to find a series of disjoint balls $B_i$ of radiuses $a_i$ for any sequence $a_i$ as above, provided that $\e$ is small enough (because $\mathcal H^1(\Omega)=\infty$ and for any set $C$, $\mathcal H^1(C)>0$ implies $H^{2-p}(C)=\infty$). Inside each $B_i$ one can insert two disjoint balls $B_i^+, B_i^-$ of radius $\tfrac{a_i}{2}$. Call $x_i^\pm$ the center of $B_i^\pm$, and consider the current 
$$
I=\sum_{i=1}^{\infty}[x_i^-, x_i^+].
$$
Using the estimate \eqref{infinitest} and the estimates \eqref{lpnormestimate} on the disjoint balls $B_i^\pm$, we obtain that any vectorfield satisfying $\op{div}V=\de I$ must not be in $L^p$. By our Main Theorem (second version), we see that none of the currents constructed in this way can possibly have boundary equal to the distributional Jacobian of a map $u\in W^{1,p}(\Omega,S^1)$.
\end{example}

\bibliographystyle{amsalpha}
\bibliography{integrability}
\end{document}